\documentclass[10pt, a4paper]{amsart}
\usepackage{amsthm,mathrsfs,amsmath,amscd,enumerate,enumitem, amscd, xcolor, url}
\usepackage[colorlinks]{hyperref}
\usepackage{amsfonts}
\usepackage{amssymb, mathtools}
\usepackage[utf8]{inputenc}

\newtheorem{thm}{Theorem}[section]

\newtheorem{lem}[thm]{Lemma}
\newtheorem{prop}[thm]{Proposition}
\theoremstyle{definition}
\newtheorem{defn}[thm]{Definition}
\theoremstyle{remark}
\newtheorem{rem}[thm]{Remark}

\DeclareMathOperator{\supp}{supp}
\DeclareMathOperator{\Dom}{Dom}

\newcommand{\vertiii}[1]{{\left\vert\kern-0.25ex\left\vert\kern-0.25ex\left\vert #1 
    \right\vert\kern-0.25ex\right\vert\kern-0.25ex\right\vert}}

\numberwithin{equation}{section}
\allowdisplaybreaks

\usepackage[a4paper, left=2.5cm, right=2.5cm, top=3cm, bottom=3cm]{geometry} 

\begin{document}
    \title[Higher order Gaussian Riesz transforms]
    {Endpoint estimates for higher order Gaussian Riesz transforms}

    \author[F. Berra]{Fabio Berra}
    \address{Fabio Berra\newline
        Universidad Nacional del Litoral, FIQ. Research of CONICET.\newline Santiago del Estero 2829,\newline S3000AOM, Santa Fe, Argentina}
    \email{fabiomb08@gmail.com}

    \author[E. Dalmasso]{Estefan\'ia Dalmasso}
    \address{Estefan\'ia Dalmasso\newline
        Instituto de Matem\'atica Aplicada del Litoral, UNL, CONICET, FIQ.\newline Colectora Ruta Nac. Nº 168, Paraje El Pozo,\newline S3007ABA, Santa Fe, Argentina}
    \email{edalmasso@santafe-conicet.gov.ar}

    \author[R. Scotto]{Roberto Scotto}
    \address{Roberto Scotto\newline
        Universidad Nacional del Litoral, FIQ.\newline Santiago del Estero 2829,\newline S3000AOM, Santa Fe, Argentina}
    \email{roberto.scotto@gmail.com}

    \thanks{The authors were supported by PICT 2019-0389 (ANPCyT), CAI+D 2020 50320220100210 (UNL) and PIP 11220200101916CO (CONICET)}
    \date{\today}


    \begin{abstract}
    We will show that, contrary to the behavior of the higher order Riesz transforms studied so far on the atomic Hardy space $\mathcal{H}^1(\mathbb R^n, \gamma)$, associated with the Ornstein-Uhlenbeck operator with respect to the $n$-dimensional Gaussian measure $\gamma$, the new Gaussian Riesz transforms are bounded from $\mathcal{H}^1(\mathbb R^n, \gamma)$ to $L^1(\mathbb R^n, \gamma)$, for any order and dimension $n$. We will also prove that the classical Gaussian Riesz transforms of higher order are bounded from an adequate subspace of $\mathcal{H}^1(\mathbb R^n, \gamma)$ into $L^1(\mathbb R^n, \gamma)$, extending Bruno's result (J. Fourier Anal. Appl. 25, 4 (2019), 1609--1631) for the first order case. 
    \end{abstract}
    \date{\today}
    \maketitle

\section{Introduction}

For $x\in \mathbb{R}^n,$ let $d\gamma(x)=\pi^{-n/2}e^{-|x|^2}dx$ be the $n$-dimensional non-standard Gaussian measure and let~$\mathcal{L}$ be the closure  on $L^2(\gamma)$ of the Ornstein-Uhlenbeck differential operator given by
\[L=-\frac{1}{2}\Delta +x\cdot \nabla=\sum_{i=1}^n \delta^*_i \delta_i,\]
where $\delta_i=\frac{1}{\sqrt{2}}\frac{\partial}{\partial x_i}$, and $\delta^*_i =-\frac{1}{\sqrt{2}}e^{|x|^2}\frac{\partial }{\partial x_i}\left(e^{-|x|^2}\cdot\right)$ is the formal adjoint of $\delta_i$ on $L^2(\gamma).$ This operator $L$ is 
 defined on the space $C_c^\infty(\mathbb{R}^n)$ of smooth and compactly supported functions on $\mathbb{R}^n.$ It is well known that $\mathcal{L}$ is an unbounded positive self-adjoint operator on $L^2(\gamma)$. Its spectrum is discrete composed by non-negative integers as eigenvalues whose eigenfunctions are the normalized $n$-dimensional Hermite polynomials $\{h_\alpha\}_{\alpha\in \mathbb{N}_0^n}$ which turn out to be an orthonormal basis on $L^2(\gamma).$ That is, $\mathcal{L}h_\alpha=|\alpha|h_\alpha$ with $|\alpha|=\sum_{i=1}^n\alpha_i$ and $\textup{Dom}(\mathcal{L})=\{f\in L^2(\gamma):\sum_{\alpha\in \mathbb{N}_0^n}|\alpha|^2| \langle f, h_\alpha\rangle|^2<\infty\}.$ Let us observe also that $C_c^\infty(\mathbb{R}^n)\subseteq \textup{Dom}{(\mathcal{L})}$ and $\mathcal{L}=L$ on $C_c^\infty (\mathbb{R}^n).$

We define two types of higher order Gaussian Riesz transforms, known in the literature as the ``old'' and the ``new'' ones. First, let us note that these transforms can be spectrally defined as follows: for a multi-index ${\alpha\in \mathbb{N}_0^n}\setminus \{(0,\dots,0)\}$, the ``old'' Gaussian Riesz transforms of order $\alpha$ are given by
\[R_\alpha = D^\alpha \mathcal{L}^{-\frac{|\alpha|}{2}}\] being $D^\alpha=\delta_1^{\alpha_1}\ldots \delta_n^{\alpha_n};$ and the ``new'' ones have the form
\[R_\alpha^*=D^{*\alpha}(\mathcal{L}+I)^{-\frac{|\alpha|}{2}}\] being $D^{*\alpha}=\delta_1^{*\alpha_1}\ldots \delta_n^{*\alpha_n}$ and $I$ the identity operator on $L^2(\gamma).$  

Our main aim is to analyze the continuity of these singular integrals on $L^1(\gamma).$ It is well known, as it happens in the classical context for the Laplacian operator, that these transforms are not bounded on $L^1(\gamma).$ And as it is noted in \cite{MMS13} the ``old'' Riesz transforms are not bounded neither from the Gaussian atomic Hardy space $\mathcal{H}^1(\gamma)$ (given in \cite{MM07}) into $L^1(\gamma),$ for $n>1.$ There was a satisfactory answer of this matter, provided by T. Bruno in \cite{Bruno}, in the case of the first order old Gaussian Riesz transforms, which are bounded from a certain subspace $X^1(\gamma)\subseteq  \mathcal{H}^1(\gamma)$ into $L^1(\gamma)$. Now we shall complete this study for higher order old Gaussian Riesz transforms by using smaller subspaces as we increase the order of the transform, and at the same time we prove the boundedness of the new ones from $\mathcal{H}^1(\gamma)$ into $L^1(\gamma).$

To that end, we will first need the concept of atom to introduce the corresponding atomic Hardy space $\mathcal{H}^1(\gamma)$  later.
Given $r\in (1,\infty],$ a Gaussian $(1,r)$--atom is either the constant function $1$ or
a function $a \in L^r(\gamma)$ supported in an admissible ball $B$ (see its definition in Section~\ref{seccion: preliminares}) such that
\begin{equation*}
\int a\, d\gamma=0 \quad\text{and}\quad  \|a\|_r\le \gamma(B)^{\frac{1}{r}-1}.
\end{equation*}
From now on, the symbol $\|\cdot\|_r$ denotes the norm in $L^r(\gamma).$ In the latter case, we
say that the atom $a$ is associated to the ball $B.$

The space $\mathcal{H}^{1,r}(\gamma)$ is then the vector space of all functions $f \in L^1(\gamma)$ that admit a decomposition of the form
$\sum_j \lambda_j a_j$,
where the $a_j$ are Gaussian $(1,r)$--atoms and the series associated to the sequence of complex numbers $\{\lambda_j\}$
is absolutely convergent. The norm of $f$ in $\mathcal{H}^{1,r}(\gamma)$ is defined as the infimum of $\sum_j |\lambda_j|$ over
all these representations of $f$.

In \cite{MM07} and \cite{MMS13} the spaces $\mathcal{H}^{1,r}(\gamma)$ were defined and proved to coincide for all $1 < r \le \infty,$ 
with equivalent norms. So any one of them is called $\mathcal{H}^1(\gamma).$ In view of these facts, we shall refer to the atoms in $\mathcal{H}^1(\gamma)$ as $\mathcal{H}^1$--atoms.

In Section \ref{seccion: preliminares} we will introduce a new Hardy space $X^k(\gamma)$ for each $k\in \mathbb{N}$. The sequence of these new Hardy spaces form a strictly decreasing chain such that $X^{k+1}(\gamma)\subsetneq X^{k}(\gamma)\subsetneq \mathcal{H}^1(\gamma).$ And these spaces will be suitable to the boundedness of $R_\alpha$ when $p=1$, where the parameter $k$ is related to the order $\alpha$ as we can see below.

We state our two main theorems which take care of the boundedness of these Gaussian Riesz transforms on $L^1(\gamma).$

\begin{thm}\label{thm: teo principal}
$R_\alpha$ is bounded from $X^k(\gamma)$ to $L^1(\gamma)$, for any multi-index $\alpha$ with $k=|\alpha|$, and any dimension.
\end{thm}

\begin{thm}\label{thm: teo principal 2}
$R_\alpha^*$ is bounded from $\mathcal{H}^1(\gamma)$ to $L^1(\gamma)$, for any multi-index $\alpha$ and any dimension. 
\end{thm}

The article is organized as follows. In Section~\ref{seccion: preliminares} we introduce some notation, definitions and properties of $\mathcal{L}$, whereas in Section~\ref{seccion: atomos y hardys} we establish the definitions of atoms and Hardy type spaces in the Gaussian framework. In order to prove our main theorems in Section~\ref{seccion: resultados principales}, we give previously several auxiliary results in Section~\ref{seccion: auxiliares}.

\section{Preliminaries} \label{seccion: preliminares}

For the operator $\mathcal{L}$ introduced before, and for every $z\in \mathbb{C}$ we define 
\begin{align*}
\mathcal{L}^z=\sum_{j=1}^\infty j^z \mathcal{P}_j, &\ \  \Dom(\mathcal{L}^z)=\left\{f\in L^2(\gamma):\sum_{j=1}^\infty j^{2\text{Re}\, z} \|\mathcal{P}_j f\|_2^2<\infty\right\},
\end{align*}
being $\mathcal{P}_j$ the orthogonal projection onto the Wiener chaos space of order $j$, i.e. $\mathcal{P}_j f=\sum_{|\alpha|=j} \langle f, h_\alpha\rangle h_\alpha$ with $h_\alpha $ the normalized Hermite polynomial of degree~$|\alpha|=j,$ where $\alpha\in \mathbb N_0^n$. Let us remark that $\mathcal{L}^1=\mathcal{L}.$

For $f\in \Dom(\mathcal{L}),$ we have
\[\mathcal{L}f(x)=\sum_{j=1}^\infty j\, \mathcal{P}_j f(x).\]
 Recall that the family of orthonormalized Hermite polynomials $\{h_\alpha\}_{\alpha\in \mathbb N_0^n}$ is an orthonormal basis on $L^2(\gamma)$ (see, for instance, \cite{Urbina}).
Thus, we can rephrase \[\Dom(\mathcal{L})=\left\{f\in L^2(\gamma): \mathcal{L}f\in L^2(\gamma)\right\}.\]

If $\text{Re}\, z<0,$ the operator $\mathcal{L}^z$ turns out to be bounded on $L^2(\gamma)$ and we have $\Dom(\mathcal{L}^z)=L^2(\gamma).$ Meanwhile, if $\text{Re}\, z\ge 0$, then $C_c^\infty(\mathbb{R}^n)\subseteq  \Dom(\mathcal{L}^z)$ by the decomposition $\mathcal{L}^z= \mathcal{L}^{z-N}\mathcal{L}^N$ with ${N=[\text{Re}\, z]+1.}$

Let $\Pi_0$ be the orthogonal projection 
\[\Pi_0 : L^2(\gamma)\to \text{ker}(\mathcal{L})^\perp=\left\{f\in L^2(\gamma): \int_{\mathbb{R}^n} f d\gamma=0\right\}.\]

In terms of the spectral resolution $\{\mathcal{P}_j\}_{j\ge 0},$ $\Pi_0=I-\mathcal{P}_0$ since \[\mathcal{P}_0:L^2(\gamma)\to \text{ker}(\mathcal{L})\equiv{\mathbb{C}},\]
with \[\mathcal{P}_0 f=\int_{\mathbb{R}^n} f\, d\gamma.\]

We shall denote the space $\Pi_0(L^2(\gamma))$ as $L_0^2(\gamma).$

\begin{lem}
For every  positive integer $k$ we have
\[\mathcal{L}^k\mathcal{L}^{-k}f=\Pi_0 f \quad \forall f\in L^2(\gamma), \quad \mathcal{L}^{-k}\mathcal{L}^k f=\Pi_0 f\quad \forall f \in \Dom(\mathcal{L}^k).\]
\end{lem}

\begin{proof}
 Let us see that for $f\in L^2(\gamma),$ 
$\mathcal{L}^{-k}f\in \Dom(\mathcal{L}^k).$ From the definitions of $\mathcal{P}_j$ and $\mathcal{L}^{-k}$, and the orthonormality of $\{h_\alpha\}_{\alpha\in \mathbb N_0^n}$ we have 
\[\mathcal{P}_j(\mathcal{L}^{-k}f)= \frac{1}{j^k}\mathcal{P}_j f\] and 
thus 
\begin{align*}\sum_{j= 1}^\infty j^{2k}\|\mathcal{P}_j (\mathcal{L}^{-k}f)\|_2^2 & =\sum_{j= 1}^\infty j^{2k}\frac{1}{j^{2k}}\|\mathcal{P}_j f\|_2^2\\ & =\sum_{j= 1}^\infty\|\mathcal{P}_j f\|_2^2\\ & \le \sum_{j= 0}^\infty\|\mathcal{P}_j f\|_2^2\\
&=\|f\|_2^2<\infty.
\end{align*}
Hence, $\mathcal{L}^{-k}f\in \Dom(\mathcal{L}^k)$ and from the definition of $\mathcal{L}^k$ and $\mathcal{L}^{-k}$ we get that, for $f\in L^2(\gamma)$,
\[\mathcal{L}^k \mathcal{L}^{-k}f=\sum_{j=1}^\infty \mathcal{P}_j f= f-\int_{\mathbb{R}^n} f\, d\gamma=\Pi_0 f.\]

Similarly, 
for $f\in \Dom(\mathcal{L}^k),$ we get  \[\mathcal{L}^{-k}\mathcal{L}^k f=\Pi_0 f.\qedhere\]
\end{proof}
Particularly, from the above result, we have
\begin{equation*}
\mathcal{L}^k \mathcal{L}^{-k} f=f\quad \forall f\in L_0^2(\gamma),\quad \mathcal{L}^{-k} \mathcal{L}^k f= f\quad \forall f\in \Dom(\mathcal{L}^k)\cap L_0^2(\gamma).
\end{equation*}

Let us introduce several function spaces that will play an important role in the definition of special atoms and investigate the relationship among them.

Recall that $\mathcal{L}$ is an elliptic operator. Thus, for a given bounded open subset $\Omega$ of $\mathbb{R}^n,$ a positive integer~$k$ and a real constant $c,$ every solution $u$ of the equation
\[\mathcal{L}^k u=c\mathcal{X}_\Omega\]
is smooth in $\Omega$ by elliptic regularity. Here, $\mathcal{X}_\Omega$ denotes the characteristic function of $\Omega.$ 

We will now introduce some function spaces related with the solutions of the integer powers of $\mathcal{L}$.

\begin{defn} Suppose that $k$ is a positive integer and that $\Omega$ is a bounded open subset of $\mathbb{R}^n.$ We say that a function $u$ is $k$--\emph{quasi-harmonic} on $\Omega$ if $\mathcal{L}^k u$ is constant on $\Omega$ (in the sense of distributions, hence in the classical sense, since $u$ is smooth by elliptic regularity). We shall denote by $q_k^2(\Omega)$
the space of $k$--quasi-harmonic functions on $\Omega$ which belong to $L^2(\gamma)$.

The subspace of $q_k^2(\Omega)$ of all functions $v$ such that $\mathcal{L}^kv=0$ on $\Omega$ will be denoted by $h_k^2(\Omega).$
\end{defn}

\begin{defn}
Given a compact subset $K\subseteq  \mathbb{R}^n$, we say that a function $v$ is $k$--\emph{quasi-harmonic} on $K$ if $v$ is the restriction to $K$ of a function in $q_k^2(\Omega)$ for some bounded open set $\Omega$ containing $K.$ We shall denote by $q_k^2(K)$ the space of all $k$--quasi-harmonic functions on $K.$ The subspace of all functions which are restrictions to $K$ of functions in $h_k^2(\Omega)$ will be denoted by $h_k^2(K).$ 
\end{defn}

\begin{rem}\label{remark: caracterizacion q}
Let $K$ be a compact subset of $\mathbb{R}^n$. Then
$q_k^2(K)^\bot=\left\{v\in L^2(\gamma): \mathcal{L}^{-k}v\in L^2_0(K,\gamma)\right\}$ being $L^2_0(K,\gamma)$ the space of functions $w\in L^{2}(\gamma)$ with $\supp w\subseteq  K$ and $\int w d\gamma=0.$
\end{rem}
The proof of this remark can be found in \cite{MMV12} for the  Laplace--Beltrami operator and $K=\overline{B}$ with $B$ a ball. Although it follows the same lines as the ones in \cite[Proposition~3.3(i)]{MMV12}, we will state it here for the sake of completeness.
\begin{proof}[Proof of Remark~\ref{remark: caracterizacion q}]
First let us prove that $q_k^2(K)^\bot$ is contained in $\{v\in L^2(\gamma): \mathcal{L}^{-k}v\in L^2_0(K,\gamma)\}.$
Let $v\in q_k^2(K)^\bot.$ In order to prove that the support of $\mathcal{L}^{-k}v$ is a subset of $K$ it suffices to show that 
$\langle \mathcal{L}^{-k}v,\mathcal{X}_{B}\rangle=0$ for every ball $B\subseteq \mathbb{R}^n\setminus K.$ 
Since $\mathcal{L}$ is self-adjoint then \[\langle \mathcal{L}^{-k}v,\mathcal{X}_B\rangle=\langle v, \mathcal{L}^{-k}\mathcal{X}_B\rangle .\] On the other hand, notice that $\mathcal{L}^{-k}\mathcal{X}_B\in {q}_k^2(K).$ Indeed, there exists a bounded open subset $\Omega$ of $\mathbb{R}^n$ with $K\subseteq  \Omega$ and $\mathcal{L}^k \mathcal{L}^{-k}\mathcal{X}_{B}=\mathcal{X}_B - \gamma (B)$ on $\Omega$; in particular, with $\mathcal{X}_B=0$ on $K.$ Thus the last inner product in the above equality vanishes. 

Now we prove that the function $\mathcal{L}^{-k}v$ has average $0$ with respect to $\gamma.$ Indeed, since $\supp( \mathcal{L}^{-k}v)\subseteq  K$ and $\mathcal{L}$ is self-adjoint, we have
\[\int_{\mathbb{R}^n}\mathcal{L}^{-k}v\, d\gamma=\langle \mathcal{L}^{-k}v,\mathcal{X}_{\Omega}\rangle=\langle v,\mathcal{L}^{-k}\mathcal{X}_{\Omega}\rangle.\]
Since $v\in q_k^2(K)^\bot$ and $\mathcal{L}^{-k}\mathcal{X}_\Omega\in q^2_k(K)$ then this last integral vanishes, as required. 

Next let us prove the other inclusion. We assume that for $v\in L^2(\gamma),$ $\mathcal{L}^{-k}v\in L_0^2(K,\gamma).$ Then $v\in \Dom (\mathcal{L}^k)$ and $v=\mathcal{L}^k \mathcal{L}^{-k}v.$ If we take $w\in q_k^2(K)$, there exists $\Omega$ a bounded open subset of $\mathbb{R}^n$ with $K\subseteq  \Omega$ such that $\mathcal{L}^kw$ is constant on $\Omega.$ Therefore $w$ turns out to be smooth on $\Omega,$ and 
\[\langle v, w\rangle = \langle \mathcal{L}^k\mathcal{L}^{-k}v, w\rangle= \langle \mathcal{L}^{-k}v, \mathcal{L}^k w\rangle=0.\]
The last equality is due to the fact that $\mathcal{L}^k w$ is constant on $\Omega$ and $\mathcal{L}^{-k}v\in L^2_0(K,\gamma).$   
\end{proof}

Everything that is said henceforth is understood in the sense of distributions, unless something else is specified. Let $u$ be a function in $\Dom(\mathcal{L}^k)$
such that it vanishes on the complement of $\overline{B}$ for $B$ a ball in $\mathbb{R}^n$. Then $\mathcal{L}^k u$ is in $L^2(\gamma)$ and vanishes on $\mathbb{R}^n\setminus \overline{B}.$
For every ball $B$ we introduce two operators $\mathcal{L}_B^k$ and $\mathcal{L}_{B,\textup{Dir}}^k$ defined as the restriction of $\mathcal{L}^k$ (in the distribution sense) to
\[\Dom(\mathcal{L}_B^k):= \{u\in \Dom(\mathcal{L}^k):\supp  u\subseteq  \overline{B}\},\]
\[\Dom(\mathcal{L}_{B,\textup{Dir}}^k):=\left\{u\in W_0^{2k-1,2}(B): \mathcal{L}^k u\in L^2(\gamma),\, \supp(\mathcal{L}^ku)\subseteq  B\right\},\]
respectively. Here, for $j\in \mathbb N$, we denote by $W_0^{j,2}(B)$ the closure of $C_c^\infty(B)=\{u\in C_c^\infty(\mathbb{R}^n): \supp u\subseteq B\}$ with respect to the norm 
\[\|u\|_{W_0^{j,2}(B)}=\left(\sum_{|\alpha|\leq j} \|D^\alpha u\|_{L^2(B)}^2\right)^{1/2}.\]
Notice that $W_0^{j,2}(B)=W_0^{j,2}(B,\gamma)$  since $B$ is a bounded set.

The following lemma gives an identification of the domain of $\mathcal{L}_B^k$, for any $k\in \mathbb N$. The case $k=1$ can be found in \cite[Lemma~2.6]{Bruno}.

\begin{lem}
Let $B$ be a ball in $\mathbb{R}^n.$ Then $\Dom(\mathcal{L}_B^k)=W_0^{2k,2}(B)$ with equivalence of norms.
\end{lem}

\begin{proof}
As said above, for the case $k=1$ we have the result by T. Bruno in \cite{Bruno}. For the general case we use induction on $k$ since $\Dom(\mathcal{L}_B^k)=\{u\in \Dom(\mathcal{L}_B): \mathcal{L} u\in \Dom(\mathcal{L}_B^{k-1})\}.$ Indeed, let us assume that $\Dom(\mathcal{L}_B^k)=W_0^{2k,2}(B).$ Then  $\Dom(\mathcal{L}_B^{k+1})=\{u\in \Dom(\mathcal{L}_B): \mathcal{L} u\in \Dom(\mathcal{L}_B^{k})=W_0^{2k, 2}(B)\}.$ Thus, given $u\in \Dom(\mathcal{L}_B^{k+1})$, we have $\mathcal{L}u=f$ with $f\in W_0^{2k,2}(B)$ and $\supp u\subseteq  \overline{B},$ so we get $u\in W_0^{2k+2,2}(B)$ (cf. \cite[Theorem 6.5]{Evans}). Conversely, taking into account that $D^\alpha L=L D^\alpha - 2|\alpha|D^\alpha$ and that $B$ is a bounded set, if $u\in W_0^{2k+2,2}(B)$ then $\mathcal{L} u\in W_0^{2k,2}(B)$, from this $u\in \Dom(\mathcal{L}_B^{k+1})$. So the case $k+1$ is also true.  
\end{proof}

In the spirit of \cite[Lemma~2.7]{Bruno} we can obtain the following lemma regarding several properties for the functions spaces previously defined, as well as the operators $\mathcal{L}_B^k$ and $ \mathcal{L}_{B,\textup{Dir}}^k$.

\begin{lem}\label{lema: propiedades de L^k, h_k^2 y q_k^2}
The following statements hold. 
\begin{enumerate}[label=(\roman*)]
\item \label{item: lema: propiedades de L^k, h_k^2 y q_k^2 - item 1}Both spaces $q_k^2(B)$ and $h_k^2(B)$ are closed subspaces of $L^2(B)$;
\item \label{item: lema: propiedades de L^k, h_k^2 y q_k^2 - item 2}$\mathcal{L}^k$ is a Banach space isomorphism between $\textup{Dom}(\mathcal{L}_B^k)$ and $h^{2}_k(\overline{B})^{\perp}$;
\item \label{item: lema: propiedades de L^k, h_k^2 y q_k^2 - item 3}$h_k^2(\overline{B})^{\perp}=h_k^2(B)^{\perp}$;
\item \label{item: lema: propiedades de L^k, h_k^2 y q_k^2 - item 4}$\textup{Ran}(\mathcal{L}_B^k)=h_k^2(B)^{\perp}$;
\item \label{item: lema: propiedades de L^k, h_k^2 y q_k^2 - item 5}$q_k^2(\overline{B})^{\perp}=q_k^2(B)^{\perp}$;
\item \label{item: lema: propiedades de L^k, h_k^2 y q_k^2 - item 6}$\textup{Dom}(\mathcal{L}_B^k)\subseteq \textup{Dom}(\mathcal{L}_{B,\textup{Dir}}^k)$.
\end{enumerate}
\end{lem}

\begin{proof}
Let us first prove item~\ref{item: lema: propiedades de L^k, h_k^2 y q_k^2 - item 1}. Clearly $h_k^2(B)$ is a subspace of $q_k^2(B)$ of co-dimension one. Indeed, we have the vector space decomposition 
\[q_k^2(B)=h_k^2(B)\oplus \mathbb{C}\left(\mathcal{L}^{-k}\psi|_B\right),\] being $\psi\in L_0^2(\gamma)\cap C_c^\infty (\mathbb{R}^n)$ such that $\psi=1$ on $B$, and $\mathbb{C}(\varphi)=\{c\varphi: c\in \mathbb{C}\}$. Let $u\in q_k^2(B)$. Then $\mathcal{L}^k u= c=c\psi $ on $B$. From the definition of $\psi$ we have $\mathcal{L}^k \mathcal{L}^{-k}\psi =\psi.$ If we take $v=u-c\mathcal{L}^{-k}\psi|_B$ we have that $\mathcal{L}^k v=\mathcal{L}^k u-c\mathcal{L}^k\mathcal{L}^{-k}\psi=0$ on $B$. Thus  $v\in h_k^2(B)$ and $u=v+ c \mathcal{L}^{-k}\psi|_B$. Since $\psi\in L_0^2(\gamma)$, we have that the sum is direct.

Let us prove that $q_k^2(B)$ and $h_k^2(B)$ are closed subspaces of $L^2(B).$ Due to the above decomposition it is enough to prove that $h_k^2(B)$ is closed. Indeed, let $\{v_j\}_{j\in \mathbb N}$ be a sequence on $h_k^2(B)$ that converges on $L^2(B)$ to $v\in L^2(B)$. Then $\mathcal{L}^k v_j$ converges to $\mathcal{L}^k v$ in the sense of distributions. Thus $\mathcal{L}^kv=0$ on $B,$ whence $v\in h_k^2(B)$.

We turn now our attention to item~\ref{item: lema: propiedades de L^k, h_k^2 y q_k^2 - item 2}. We first observe that $\mathcal{L}^k(\textup{Dom}(\mathcal{L}_B^k))\subseteq h_k^2(\overline{B})^\perp$. Let $f\in \textup{Dom}(\mathcal{L}_B^k)$ and fix $v\in h_k^2(\overline{B})$. Then there exist an open set $\Omega\supseteq \overline{B}$ and a function $w\in h_k^2(\Omega)$ such that $w\vert_{\overline{B}}=v$.
We can construct a function $\tilde{v}\in C_c^\infty(\mathbb{R}^n)\cap h_k^2(\mathbb{R}^n)$ such that $\tilde v=w$ on $\overline{B}$ and $\supp\,\tilde{v}\subseteq \Omega$. Observe that
\[\int_{\overline{B}} v\mathcal{L}^kf\,d\gamma=\int_{\overline{B}} w \mathcal{L}^k f\, d\gamma= \int_{\Omega} \tilde{v}\mathcal{L}^k f\, d\gamma= \int_{\mathbb{R}^n} \tilde{v}\mathcal{L}^kf\,d\gamma=\int_{\mathbb{R}^n}\mathcal{L}^k\tilde{v}\, f\,d\gamma=0,\]
where we have used that
\begin{align*}
    \int_{\mathbb{R}^n} \tilde{v}\mathcal{L}^kf\,d\gamma&=\int_{\mathbb{R}^n}\left(\sum_{\alpha}\langle \tilde{v},h_\alpha \rangle h_\alpha\right)\left(\sum_{\beta} |\beta|^k\langle f,h_\beta \rangle h_\beta\right) \,d\gamma\\
    &=\sum_\alpha\sum_\beta \langle \tilde{v},h_\alpha\rangle |\beta|^k\langle f,h_\beta\rangle\langle h_\alpha, h_\beta\rangle \\
    &= \sum_\alpha |\alpha|^k \langle \tilde{v} ,h_\alpha \rangle \langle f,h_\alpha\rangle\\
    &=\int_{\mathbb{R}^n}\mathcal{L}^k \tilde{v} f\,d\gamma.
\end{align*}

At this point we shall prove now that $\mathcal{L}_B^k$ is injective on $\textup{Dom}(\mathcal{L}_B^k)$. Indeed, if $f\in \textup{Dom}(\mathcal{L}_B^k)$ and $\mathcal{L}^k f=0$ we have that $\mathcal{L}(\mathcal{L}^{k-1}f)=0$, so $\mathcal{L}^{k-1}f$ is a constant $L^2$--function with support contained in $\overline{B}$. Therefore $\mathcal{L}^{k-1}f=0$. By proceeding recursively we get that  $f=0$.

As the next step we will prove that $\mathcal{L}^k$ maps $\textup{Dom}(\mathcal{L}_B^k)$ onto $h_k^2(\overline{B})^\perp$. Fix $v\in h_k^2(\overline{B})^\perp$ and define $\tilde v=v\mathcal{X}_{\overline{B}}$. Then $f=\mathcal{L}^{-k}(\tilde v)\in \textup{Dom}(\mathcal{L}^k)$ since $\mathcal{L}^{-k}: L^2(\gamma)\to \textup{Dom}(\mathcal{L}^k)$. We also have
\[\int_{\mathbb{R}^n}\tilde v\,d\gamma=\int_{\overline{B}} v\,d\gamma=0\] since $1\in h_k^2(\overline{B})$. Thus $\tilde v\in L_0^2(\gamma)$ and consequently 
\[\mathcal{L}^k f=\mathcal{L}^k\left(\mathcal{L}^{-k}\tilde v\right)=\Pi_0 \tilde v=\tilde v.\]
We now fix a test function $\phi\in C_c^\infty(\overline{B}^c)\cap L_0^2(\gamma)$. Then we can write
\[\langle \phi, f\rangle  =\left\langle \mathcal{L}^k(\mathcal{L}^{-k}\phi),\mathcal{L}^{-k}\tilde{v}\right\rangle=\left\langle \mathcal{L}^{-k}\phi,\tilde v\right\rangle=\int_{\overline{B}}\mathcal{L}^{-k}\phi v\,d\gamma =0,\]
since $\mathcal{L}^{-k}\phi \in h_k^2(\overline{B})$. By the arbitrariness of $\phi$ we can conclude that $f$ is constant on $\overline{B}^c$, that is, $f(x)=c$ for every $x\in \overline{B}^c$. The function $g=f-c$ belongs to $\textup{Dom}(\mathcal{L}^k)$, $\supp g\subseteq \overline{B}$ and 
\[\mathcal{L}^k g=\mathcal{L}^k(f-c)=\mathcal{L}^k f=\tilde v=v.\] 
Consequently, we have proved that the operator $\mathcal{L}^k$ is a bijection between $\textup{Dom}(\mathcal{L}_B^k)$ and $h_k^2(\overline{B})^\perp$. We also have that $\mathcal{L}^k$ is continuous since $\textup{Dom}(\mathcal{L}_B^k)=W^{2k,2}_0(\overline{B})$.

Let $T=\mathcal{L}^k: \textup{Dom}(\mathcal{L}_B^k)\to h_k^2(\overline{B})^\perp$.
We are going to show that $T^{-1}$ is continuous from $h_k^2(\overline{B})^\perp$ to $\textup{Dom}(\mathcal{L}_B^k)$. In order to prove that, we shall see that its graph is closed. Let $\{f_j\}_{j\in \mathbb N}$ be a sequence in $h_k^2(\overline{B})^\perp$ such that $f_j\to f$ when $j\to\infty$ on $L^2(\overline{B})$ and $T^{-1}f_j\to g$, where $g\in W_0^{2k,2}(\overline{B})$. By the surjectivity of $T^{-1}$ we have that there exists $\tilde f\in h_k^2(\overline{B})^\perp$ such that $g=T^{-1}\tilde f$. By the continuity of $T$ we get
\[f_j=T\left(T^{-1}f_j\right)\to Tg=\tilde f,\]
then we conclude that $f=\tilde f$. By the closed graph theorem we obtain that $T^{-1}$ is continuous.

The proof of items~\ref{item: lema: propiedades de L^k, h_k^2 y q_k^2 - item 3} to~\ref{item: lema: propiedades de L^k, h_k^2 y q_k^2 - item 6} follow similar lines as in \cite[Lemma~2.7]{Bruno} and we shall omit it.
\end{proof}

\section{Hardy type spaces and atoms}\label{seccion: atomos y hardys}

As mentioned in the introduction, the definition of the suitable atoms for the old Gaussian Riesz transforms will require to consider a family of balls 
such that the Gaussian measure $\gamma$ is doubling on such a family. 

 Given an Euclidean ball $B=B(c_B,r_B)$ where $c_B$ is its center, and $r_B>0$ its radius, we will say that $B$ is an \textit{admissible ball} if $r_B\leq m(|c_B|)$ where the function $m:\mathbb{R}^ +_0\to \mathbb{R}^+$ is defined as 
 \[m(s)=\begin{cases} 1, & 0\le s\le 1,\\ \frac1s, & s> 1.
 \end{cases}\]
The collection of these balls will be denoted by $\mathscr{B}_1$. For any constant $c>0$, by $cB$ we will denote the ball with same center as $B$ and $c$ times its radius.

As it is well-known (see, for instance, \cite[Proposition~2.1~and~Remark~2.2]{MM07}), the Gaussian measure is doubling on $\mathscr{B}_1$, that is, for any ball $B\in \mathscr{B}_1$, 
\begin{equation*}
    \gamma(2B)\leq D_\gamma \gamma(B),
\end{equation*}
 with doubling constant $D_\gamma>0$ independent of $B$.

We are now in position to define, for every $k\in\mathbb{N}$, an $X^k$--atom and its corresponding atomic Hardy type space $X^k(\gamma)$ in this setting.

\begin{defn}
Given $k\in \mathbb N$, we say that a function $a\in L^2(\gamma)$ is an $X^k$--atom if $a$ is supported on an admissible ball $B\in \mathscr{B}_1$, and verifies
\begin{enumerate}[label=(\alph*)]
    \item \label{item: def: condicion 1} $\|a\|_{L^2(\gamma)}\leq \omega_k(r_B)\gamma(B)^{-1/2}$, where $\omega_1(r_B)=\omega_2(r_B)=1$ and $\omega_k(r_B)=r_B^{k-2}$ for $k\geq 3$;
    \item \label{item: def: condicion 2} $a\in q_k^2(\overline{B})^\bot$. 
\end{enumerate}
\end{defn}

\begin{rem}\label{obs: acotacion norma dos de atomo mas fuerte}
Observe that condition \ref{item: def: condicion 2} implicitly says that $\int a\, d\gamma=0$ since $\mathcal{X}_{2B}\in q_k^2(\overline{B})$. On the other hand, since $\omega_k(r_B)\le 1$, for every $k$, we get from~\ref{item: def: condicion 1} that 
\begin{equation}\label{eq: old a condition}
\|a\|_{L^2(\gamma)}\le \gamma(B)^{-1/2}.
\end{equation}
\end{rem}

\begin{rem}\label{obs: contencion X^k atomos} It is easy to see that every $X^k$--atom is also an $X^j$--atom for every $1\leq j\leq k$.
\end{rem}

\begin{rem}
Condition~\ref{item: def: condicion 1} implies that each atom is in $L^1(w_k\gamma)$ with 
$$w_k(x)=\begin{dcases}
    1 & k=1,2;\\
    1+|x|^{k-2} & k\geq 3.
\end{dcases}$$ 
Indeed, for any $k\geq 1$,
\begin{equation}\label{eq: cota norma 1 pesada atomo}
    \|a\|_{L^1(w_k\gamma)}\le \|w_k\mathcal{X}_{B}\|_{L^2(\gamma)}\|a\|_{L^2(\gamma)}\le \|w_k\mathcal{X}_{B}\|_{L^2(\gamma)}\omega_k(r_B)\gamma(B)^{-1/2}\le 1+2^{k-2}.
\end{equation}

This will allow us to deduce the boundedness of the old Gaussian Riesz transforms just having proved the uniform boundedness of them on the atoms and taking into account that $R_\alpha$ with $|\alpha|=k$ are bounded from $L^1(w_k \gamma)$ into $L^{1,\infty}(\gamma)$ (see \cite{FHS09}). Let us also remark that in order to get the uniform boundedness of these old Gaussian Riesz transforms on atoms we will use the  weaker condition~\eqref{eq: old a condition}. 
\end{rem}

\begin{rem}
    We shall point out that if $a$ is an $X^k$--atom, then $\frac{a}{\vertiii{\mathcal{L}^{-k}}_2}$ is an $\mathcal{H}^1$--atom supported in $\overline{B}.$ Here $\vertiii{\, \cdot\, }_2$  stands for the norm on the Banach space of bounded linear operators defined on $L^2(\gamma).$
    
    Indeed, since $a\in q_k^2(\overline{B})^\bot,$ by Remark~\ref{remark: caracterizacion q}, $\mathcal{L}^{-k}a\in L^2_0(\overline{B},\gamma).$ Moreover, by \eqref{eq: old a condition} \[\|\mathcal{L}^{-k}a\|_{L^2(\gamma)}\le \vertiii{\mathcal{L}^{-k}}_2\, \|a\|_{L^2(\gamma)}\le \vertiii{\mathcal{L}^{-k}}_2\, \gamma(B)^{-1/2}.\]
\end{rem}

Since we are expecting a strong endpoint type result for the old Gaussian Riesz operator, we shall consider a subset of $L^1(w_k\gamma)$, given below.

\begin{defn}
Given $k\in \mathbb N$, the Hardy space $X^k(\gamma)$ is defined by
\[X^k(\gamma):=\left\{f\in L^1(\gamma): f=\sum_{j\in \mathbb N} \lambda_j a_j,\ a_j\text{ an }X^k\text{--atom }\forall j\in \mathbb N, \{\lambda_j\}_{j\in \mathbb N}\in \ell^1\right\}.\]
The norm for this space is given by
\[\|f\|_{X^k(\gamma)}:=\inf \left\{\|\{\lambda_j\}\|_{\ell^1}:f=\sum_{j\in \mathbb N} \lambda_j a_j,\ a_j\text{ an }X^k\text{--atom }\forall j\in \mathbb N \right\}.\]
\end{defn}

As claimed,  $X^k(\gamma)\subseteq L^1(w_k \gamma)$, since  
 for $f\in X^k(\gamma)$ and from \eqref{eq: cota norma 1 pesada atomo} we get 
\[\|f\|_{L^1(w_k\gamma)}\le \sum_{i=1}^\infty |\lambda_i| \|a_i\|_{L^1(w_k\gamma)}\le \left(1+2^{k-2}\right)\sum_{i=1}^\infty |\lambda_i|.\]

\section{Auxiliary results}\label{seccion: auxiliares}

We shall see next that the operator $\mathcal{L}^{-k}$ preserves the support of $X^k$--atoms, which will be one of the key ingredients in the proof of Theorem~\ref{thm: teo principal} in the following section. We refer to \cite[Proposition~2.5]{Bruno} for the case $k=1$.

In the following, by $A\lesssim B$ we shall understand that there exists a positive constant $C$ such that $A\leq C B$, where $C$ may change in each occurrence.

\begin{prop}\label{propo: soporte de L-k y su norma 2}
Let $a$ be an $X^k$--atom supported on an admissible ball $B$. Then, $\supp(\mathcal{L}^{-k}a)\subseteq \overline{B}$ and 
\[\|\mathcal{L}^{-k}a\|_{L^2(\gamma)}\lesssim r_B^{2k}\gamma(B)^{-1/2}.\]
\end{prop}

\begin{proof}
Let $a$ be an $X^k$--atom. By items~\ref{item: lema: propiedades de L^k, h_k^2 y q_k^2 - item 4} and~\ref{item: lema: propiedades de L^k, h_k^2 y q_k^2 - item 5} of Lemma~\ref{lema: propiedades de L^k, h_k^2 y q_k^2} we obtain that
\[a\in q_k^2(\overline{B})^\perp=q_k^2(B)^\perp\subseteq h_k^2(B)^\perp=\textup{Ran}(\mathcal{L}_B^k).\]
Then there exists $f\in \textup{Dom}(\mathcal{L}_B^k)\subseteq \textup{Dom}(\mathcal{L}_{B,\textup{Dir}}^k)$, by virtue of item~\ref{item: lema: propiedades de L^k, h_k^2 y q_k^2 - item 6} of Lemma~\ref{lema: propiedades de L^k, h_k^2 y q_k^2}, such that $\mathcal{L}_{B,\textup{Dir}}^k f=\mathcal{L}_{B}^k f=a$. We have that $f=\mathcal{L}^{-k}_{B,\textup{Dir}}\,a$ since $\mathcal{L}^{k}_{B,\textup{Dir}}$ is one-to-one in $\textup{Dom}(\mathcal{L}_B^k)$. Consequently, $\supp\,\mathcal{L}_{B,\textup{Dir}}^{-k}\,a= \supp\, f\subseteq \overline{B}$ and further $\mathcal{L}^{-k}_{B,\textup{Dir}}\,a=\mathcal{L}^{-k}a$. Therefore, since $\mathcal{L}_{B,\textup{Dir}}$ has a discrete spectrum (cf. \cite[Theorem~10.13]{Grigo09}), so has $\mathcal{L}_{B,\textup{Dir}}^k$ and thus
\[\left\|\mathcal{L}^{-k}a\right\|_{L^2(\gamma)}=\left\|\mathcal{L}_{B,\textup{Dir}}^{-k}a\right\|_{L^2(\gamma)}\leq \left(\lambda^\gamma_{\textup{Dir}}(B)\right)^{-k}\|a\|_{L^2(\gamma)}\leq \left(\lambda^\gamma_{\textup{Dir}}(B)\right)^{-k}\gamma(B)^{-1/2},\]
where we have used \eqref{eq: old a condition} and $\lambda^\gamma_{\textup{Dir}}(B)$ denotes the first eigenvalue of $\mathcal{L}_{B,\textup{Dir}}$.
From this point on, we can proceed as in the proof of \cite[Proposition~2.5]{Bruno}.
\end{proof}

In order to prove Theorem~\ref{thm: teo principal} we will require the following lemma, proved in \cite{MM07}.

\begin{lem}[{\cite[Lemma~7.1]{MM07}}]\label{lem: lema para teo principal - 1}
Let $B=B(c_B, r_B)$ be a ball in $\mathbb{R}^n$. For $y\in B$ set
$r_{B,y}=\frac{r_B}{2|y|}$ for $y\neq 0$, and $r_{B,y}=\infty$ for $y=0$.
\begin{enumerate}[label=(\roman*)]
    \item \label{item: lem: lema para teo principal - 1 - item 1} If $r_{B,y}\geq 1$, then $4|x-ry|\geq |x-c_B|$, for every $r\in[0,1]$ and every $x\in (2B)^c$;
    \item \label{item: lem: lema para teo principal - 1 - item 2}if $r_{B,y}< 1$, then $4|x-ry|\geq |x-c_B|$, for every $r\in[1-r_{B,y},1]$ and every $x\in (2B)^c$;
    \item \label{item: lem: lema para teo principal - 1 - item 3} for every $\delta>0$ there exist positive constants $C_1$ and $C_2$ such that
    \[\frac{1}{(1-r^2)^{n/2}}\int_{(2B)^c} e^{-\frac{\delta|x-c_B|^2}{1-r^2}}\,dx\leq C_1\varphi_\delta\left(\frac{r_B}{\sqrt{1-r^2}}\right)\leq C_1e^{-\frac{C_2r_B^2}{1-r^2}},\]
    where $\varphi_\delta(s)=(1+s)^{n-2}e^{-\delta s^2}\leq c_ne^{-C_2s^2}$, for $s>0$.
\end{enumerate}
\end{lem}

The following result will be useful in the proof of Theorem~\ref{thm: teo principal 2}, as it takes care of the derivatives of $\mathcal{L}^{k/2}$ for odd orders $k$. When $k=1$, the proof can be found in \cite[Lemma~2.8]{Bruno}.

\begin{lem}\label{lem: lema para teo principal - 2}
Let $\alpha$ be a multi-index with $|\alpha|=k$, being $k$ an odd positive integer.  Then 
\[\left\|D^\alpha \mathcal{L}^{k/2}f\right\|_{L^1((4B)^c,\gamma)}\lesssim r_B^{-2k}\|f\|_{L^1(B,\gamma)},\] 
for every ball $B\in \mathscr{B}_1$ and every $f\in L^1(\gamma)$ such that $\supp f\subseteq  \overline{B}$.
\end{lem}

\begin{proof}For $k$ odd we have that $\mathcal{L}^{k/2}$ has the following kernel (see, for instance, \cite[p.~1612]{Bruno})
\[K_{\mathcal{L}^{k/2}}(x,y)=\frac{1}{\pi^{\frac n2}\Gamma\left(-\frac k2\right)}\int_0^1 (-\log r)^{-\frac k2-1} \left(\frac{e^{-\frac{|rx-y|^2}{1-r^2}}}{(1-r^2)^{\frac n2}}-e^{-|y|^2}\right)\frac{dr}{r}.\]
For $|\alpha|=k,$ by using the chain rule and the definition of the $n$-dimensional Hermite polynomials, we have 
\[D_x^\alpha \left(e^{-\frac{|rx-y|^2}{1-r^2}}\right)=\frac{r^k}{(1-r^2)^{\frac k2}} H_\alpha \left(\frac{rx-y}{\sqrt{1-r^2}}\right)e^{-\frac{|rx-y|^2}{1-r^2}},\] 
and taking into account that $-|rx-y|^2=-|x-ry|^2+(1-r^2)\left(|x|^2-|y|^2\right)$,
we get the following kernel
\begin{align*}
    D_x^\alpha K_{\mathcal{L}^{k/2}}(x,y)&=\frac{1}{\pi^{\frac n2}\Gamma\left(-\frac k2\right)}\int_0^1 \frac{(-\log r)^{-\frac k2-1}}{(1-r^2)^{\frac n2}} D_x^\alpha \left(e^{-\frac{|rx-y|^2}{1-r^2}}\right)\frac{dr}{r}\\
    &=\frac{1}{\pi^{\frac n2}\Gamma\left(-\frac k2\right)}\int_0^1 \frac{(-\log r)^{-\frac k2-1}r^{k-1}}{(1-r^2)^{\frac{n+k}{2}}} H_\alpha \left(\frac{rx-y}{\sqrt{1-r^2}}\right)e^{-\frac{|rx-y|^2}{1-r^2}}dr\\
    &=\frac{e^{|x|^2-|y|^2}}{\pi^{\frac n2}\Gamma\left(-\frac k2\right)}\int_0^1 \frac{r^{k-1}}{(-\log r)^{\frac{k+2}{2}}(1-r^2)^{\frac{n+k}{2}}} H_\alpha \left(\frac{rx-y}{\sqrt{1-r^2}}\right)e^{-\frac{|x-ry|^2}{1-r^2}}dr.
\end{align*}

Let $\ell\in \mathbb{N}_0$ be a non-negative integer such that $k=2\ell+1$. By using again the definition of Hermite polynomials on $\mathbb{R}^n,$ it is well-known that
\[|H_\alpha (u)|\lesssim  \sum_{j=0}^{\ell} |u|^{2j+1},\] for $u\in\mathbb{R}^n.$ Thus,
\[\left|D_x^\alpha K_{\mathcal{L}^{k/2}}(x,y)\right|\lesssim \sum_{j=0}^{\ell} e^{|x|^2-|y|^2} \int_0^1 \frac{r^{k-1}|rx-y|^{2j+1}}{(-\log r)^{\frac{k+2}{2}}(1-r^2)^{\frac{n+k+1}{2}+j}} e^{-\frac{|x-ry|^2}{1-r^2}}dr.\]

Fix $f\in L^1(\gamma)$ with $\supp f\subseteq  \overline{B}$ for some ball $B=B(c_B,r_B)\in \mathscr{B}_1$. Then,
\begin{align*}
\|D_x^\alpha \mathcal{L}^{k/2}f\|_{L^1((4B)^c,\gamma)}&\leq \int_{(4B)^c} \int_{B} |D_x^\alpha K_{\mathcal{L}^{k/2}}(x,y)||f(y)| dy\,  e^{-|x|^2} dx\\
&\lesssim \sum_{j=0}^\ell \int_{(4B)^c} \int_B \int_0^1 \frac{r^{k-1}|rx-y|^{2j+1}}{(-\log r)^{\frac{k+2}{2}}(1-r^2)^{\frac{n+k+1}{2}+j}} e^{-\frac{|x-ry|^2}{1-r^2}}dr |f(y)| d\gamma(y)dx\\
&:=\sum_{j=0}^\ell \int_B I_j(y) |f(y)|d\gamma(y),
\end{align*}
where
\begin{align*}
    I_j(y)&=\int_0^1 \frac{r^{k-1}}{(-\log r)^{\frac{k+2}{2}}(1-r^2)^{\frac{n+k+1}{2}+j}} \int_{(4B)^c} |rx-y|^{2j+1} e^{-\frac{|x-ry|^2}{1-r^2}} dx dr\\
    &:=\int_0^1 J(y,r)\,dr
=\int_0^{\frac12}J(y,r)\,dr+\int_{\frac12}^1J(y,r)\,dr\\
    &:=I_{j,1}(y)+I_{j,2}(y),
\end{align*}
for each $j=0,1,\dots, \ell$.

If we prove that
\begin{equation}\label{eq: cota I_j,i}
I_{j,i}(y)\lesssim r_B^{-2k}, \quad j=0,1,\dots, \ell, \;\, i=1,2
\end{equation}
we can conclude that
\[\|D_x^\alpha \mathcal{L}^{k/2}f\|_{L^1((4B)^c,\gamma)}\lesssim\ell r_B^{-2k}\|f\|_{L^1(B,\gamma)}\approx r_B^{-2k}\|f\|_{L^1(B,\gamma)}.\]

In order to estimate \eqref{eq: cota I_j,i}  for $I_{j,1}(y)$, let us consider $v=x-ry$ so $rx-y=rv+(r^2-1)y$ and $dx=dv$. Since $y\in B\in \mathscr{B}_1$, $|y|\leq |c_B|+r_B\leq \frac{2}{r_B}$. Then
\[|rx-y|\leq r|v|+(1-r^2)|y|\leq |v|+|y|\leq |v|+\frac{2}{r_B}\leq 2\frac{|v|+1}{r_B}.\]
We apply this and use that $1-r^2\geq \frac34$ and $\frac{r^{k-1}}{(-\log r)^{\frac{k+2}{2}}}\leq \frac{2^{1-k}}{(\log 2)^{\frac{k+2}{2}}}$ for $0<r\leq \frac12$, to get
\begin{align*}
    I_{j,1}(y)&\lesssim\int_0^{\frac12} \frac{r^{k-1}}{(-\log r)^{\frac{k+2}{2}}}\int_{\mathbb{R}^n} \frac{(|v|+1)^{2j+1}}{r_B^{2j+1}} e^{-|v|^2} dv dr\\
    &\lesssim \frac{1}{r_B^{2\ell+1}} \int_{\mathbb{R}^n}e^{-|v|^2/2} dv \lesssim \frac{1}{r_B^{k}}.
\end{align*}
Clearly, since $r_B\leq 1$, we also have $I_{j,1}(y)\lesssim r_B^{-2k}$, for every $j=0,1,\dots, \ell$.

When $\frac12<r<1$, we have that $-\log r\approx 1-r^2$. By splitting \[|rx-y|=|r(x-ry)-(1-r^2)y|\leq |x-ry|+(1-r^2)|y|,\] we have
\begin{align}
    I_{j,2}(y)&\lesssim  \int_{\frac12}^1 \frac{1}{(1-r^2)^{\frac{k+2}{2}}(1-r^2)^{\frac{n+k}{2}}} \int_{(4B)^c} \left(\frac{|rx-y|}{\sqrt{1-r^2}} \right)^{2j+1} e^{-\frac{|x-ry|^2}{1-r^2}} dx dr\nonumber\\
    &\lesssim \int_{\frac12}^1 \frac{1}{(1-r^2)^{\frac n2+k+1}} \int_{(4B)^c} \left[\left(\frac{|x-ry|}{\sqrt{1-r^2}}\right)^{2j+1}+\left(\sqrt{1-r^2} |y|\right)^{2j+1}\right] e^{-\frac{|x-ry|^2}{1-r^2}} dx dr\nonumber\\
    &\lesssim \int_{\frac12}^1 \frac{1+\left(\sqrt{1-r^2}|y|\right)^{2j+1}}{(1-r^2)^{k+1}} \left(\frac{1}{(1-r^2)^{\frac n2}} \int_{(4B)^c} e^{-\frac{|x-ry|^2}{2(1-r^2)}} dx\right) dr.\label{eq: estimacion I_j,2}
\end{align}
Now, let $r_{B,y}$ be the number defined in Lemma~\ref{lem: lema para teo principal - 1} and consider the cases $r_{B,y}\geq 1$ and $r_{B,y}< 1$.

If $r_{B,y}\geq 1$, for any $x\in (4B)^c$ and $r\in \left(\frac12,1\right)$ we have that $|x-ry|\geq \frac14 |x-c_B|$ by Lemma~\ref{lem: lema para teo principal - 1}~\ref{item: lem: lema para teo principal - 1 - item 1}. Hence, from Lemma~\ref{lem: lema para teo principal - 1}~\ref{item: lem: lema para teo principal - 1 - item 3}
\[\frac{1}{(1-r^2)^{\frac n2}} \int_{(4B)^c} e^{-\frac{|x-ry|^2}{2(1-r^2)}} dx\leq \frac{1}{(1-r^2)^{\frac n2}} \int_{(4B)^c} e^{-c\frac{|x-c_B|^2}{1-r^2}} dx=C_1\varphi_c\left(\frac{r_B}{\sqrt{1-r^2}}\right)\leq C_1 e^{-C_2\frac{r_B^2}{1-r^2}},
\]
being $\varphi_c(s)=(1+s)^{n-2}e^{-cs^2}$.

This leads to
\begin{align*}
   I_{j,2}(y)&\lesssim \int_{\frac12}^1 \frac{1+\left(\sqrt{1-r^2}|y|\right)^{2j+1}}{(1-r^2)^{k+1}} e^{-C_2\frac{r_B^2}{1-r^2}} dr\\
   &\lesssim  \int_{\frac12}^1 \frac{1+\left(\sqrt{1-r^2}|y|\right)^{2j+1}}{(1-r^2)^{k-\frac12}} e^{-C_2\frac{r_B^2}{1-r^2}}\frac{2r}{(1-r^2)^{\frac32}}dr.
\end{align*}
Setting $s=r_B/\sqrt{1-r^2}$, $ds/r_B=r dr/(1-r^2)^{\frac32}$, and recalling that $r_B|y|\leq \frac12$, we get
\begin{align*}
   I_{j,2}(y)&\lesssim \int_{\frac12}^1 \frac{1+\left(\sqrt{1-r^2}|y|\right)^{2j+1}}{\left(\sqrt{1-r^2}\right)^{2k-1}} e^{-C_2\frac{r_B^2}{1-r^2}}\frac{r}{(1-r^2)^{\frac32}}dr\\
   &\lesssim \int_0^\infty \frac{1+\left(\frac{r_B}{s}|y|\right)^{2j+1}}{\left(\frac{r_B}{s}\right)^{2k-1}} e^{-C_2s^2} \frac{ds}{r_B}\\
   &=\frac{1}{r_B^{2k}} \int_0^\infty \frac{s^{2j+1}+(r_B|y|)^{2j+1}}{s^{2j+1}} s^{2k-1}e^{-C_2s^2}ds\\
   &\lesssim \frac{1}{r_B^{2k}} \int_0^\infty \left(s^{2k-1}+s^{2(k-j-1)}\right) e^{-C_2s^2} ds\\
   &\lesssim \frac{1}{r_B^{2k}},
\end{align*}
since $k-j-1\geq \frac{k-1}{2}\geq 0$ for every $0\leq j\leq \ell= \frac{k-1}{2}$. This gives estimate~\eqref{eq: cota I_j,i} when $r_{B,y}\geq 1$.

We now study the case $r_{B,y}<1$. We split the integral $I_{j,2}(y):=I_{j,2,1}(y)+I_{j,2,2}(y)$ on $\left(\frac12,1-r_{B,y}\right)$ and $(1-r_{B,y},1)$, respectively. For the second one, we can apply Lemma~\ref{lem: lema para teo principal - 1}~\ref{item: lem: lema para teo principal - 1 - item 2} and \ref{item: lem: lema para teo principal - 1 - item 3} and proceed as in the previous case.

On the interval $\left(\frac12,1-r_{B,y}\right)$, we can proceed as in \eqref{eq: estimacion I_j,2} to get
\[I_{j,2,1}(y)\lesssim  \int_{\frac12}^{1-r_{B,y}} \frac{1+(1-r^2)^{j+\frac12}|y|^{2j+1}}{(1-r^2)^{k+1}} \left(\frac{1}{(1-r^2)^{\frac n2}} \int_{(4B)^c} e^{-\frac{|x-ry|^2}{2(1-r^2)}} dx\right) dr.\]
We perform the change of variables $v=\frac{x-ry}{\sqrt{2(1-r^2)}}$, $dv=\frac{dx}{(2(1-r^2))^{\frac n2}}$, in order to obtain
\begin{align*}
    I_{j,2,1}(y)&\approx \int_{\frac12}^{1-r_{B,y}} \frac{1+(1-r^2)^{j+\frac12}|y|^{2j+1}}{(1-r^2)^{k+1}} \int_{\mathbb{R}^n} e^{-|v|^2} dv dr\\
    &\approx \int_{\frac12}^{1-r_{B,y}} \frac{1+(1-r)^{j+\frac12}|y|^{2j+1}}{(1-r)^{k+1}}dr\\
    &\leq  \frac{1}{r_{B,y}^k}+|y|^{2j+1}\int_{\frac12}^{1-r_{B,y}} (1-r)^{j-k-\frac12}dr\\
    &\lesssim \frac{1}{r_{B,y}^k}+\frac{|y|^{2j+1}}{r_{B,y}^{k-j-\frac12}}\\
    &\approx \frac{|y|^k}{r_B^k}+\frac{|y|^{k+j+\frac12}}{r_B^{k-j-\frac12}}\\
    &=\left(\frac{|y|}{r_B}\right)^k\left(1+(|y|r_B)^{j+\frac12}\right).
\end{align*}
Since $|y|\leq |c_B|+r_B\leq \frac{2}{r_B}$ and $|y|r_B\le 2$, we get
\[I_{j,2,2}(y)\lesssim r_B^{-2k}.\]
The proof is now concluded.
\end{proof}

\section{Proofs of Theorems~\ref{thm: teo principal} and~\ref{thm: teo principal 2}}\label{seccion: resultados principales}

For proving the main results, we may need the integral representations of both Gaussian Riesz transforms $R_\alpha$ and $R_\alpha^*$. Given a multi-index $\alpha\in \mathbb{N}_0^n$, we have 
\[R_\alpha f(x)=\textrm{p.v.}\int_{\mathbb{R}^n} k_\alpha(x,y)f(y)\,dy,\]
and
\[R_\alpha^* f(x)=\textrm{p.v.}\int_{\mathbb{R}^n} k_\alpha^*(x,y)f(y)\,dy,\]
where
\begin{equation*}
  k_\alpha(x,y)=c_{n,\alpha}\int_0^1 r^{|\alpha|-1}\left(\frac{-\log r}{1-r^2}\right)^{|\alpha|/2-1}H_\alpha\left(\frac{y-rx}{\sqrt{1-r^2}}\right)\frac{e^{-\frac{|y-rx|^2}{1-r^2}}}{(1-r^2)^{n/2+1}}\,dr 
\end{equation*}
and
\begin{equation*}
  k_\alpha^*(x,y)=c_{n,\alpha} e^{|x|^2-|y|^2}\int_0^1 \left(\frac{-\log r}{1-r^2}\right)^{|\alpha|/2-1}H_\alpha\left(\frac{x-ry}{\sqrt{1-r^2}}\right)\frac{e^{-\frac{|x-ry|^2}{1-r^2}}}{(1-r^2)^{n/2+1}}\,dr,
\end{equation*}
respectively.

The operator $R_\alpha$ turns out to be bounded on $L^p(\gamma),$ for $1<p<\infty,$ (see, for instance,   \cite{FSU}, \cite{GST}, \cite{Meyer}).
For the first order Gaussian Riesz transforms $R_{e_i}^*$, $i=1,\dots,n$, the $L^p(\gamma)$ boundedness was obtained in \cite{FSS} for $1<p<\infty$. By means of Meyer's multiplier theorem, the ``new'' higher order Gaussian Riesz transforms are also bounded on $L^p(\gamma)$, as can be proved similarly to \cite[Corollary 9.14]{Urbina}.

\begin{proof}[Proof of Theorem~\ref{thm: teo principal}]
  First we are going to prove that the old higher order Gaussian Riesz transforms $R_\alpha$ with $|\alpha|=k$ are uniformly bounded on $L^1(\gamma)$ when applied to  every $X^k$--atom. Once this is done the boundedness of $R_\alpha$ from $X^k(\gamma)$  into $L^1(\gamma)$ would follow from the knowledge that these transforms are bounded from $L^1(w_k\gamma)$ into $L^{1,\infty}(\gamma)$ (see \cite{FHS09}). Indeed, by following steps closely to the ones done in, for instance, \cite{BDQS2} we get for $f=\sum_j \lambda_j a_j$ with $a_j$ an $X^k$--atom and $\sum_j |\lambda_j|<\infty$ that the series defining $f$ converges on $L^1(w_k\gamma)$ and therefore $R_\alpha f=\lim_{\ell\to \infty} R_\alpha \left(\sum_{j=1}^\ell \lambda_j a_j\right)$ on $L^{1,\infty}(\gamma).$

Then there exists an increasing function  $\psi:\mathbb{N}\to \mathbb{N}$ such that $$R_\alpha f (x)=\lim_{\ell\to \infty}R_\alpha \left(\sum_{j=1}^{\psi(\ell)} \lambda_j a_j\right)(x) \ \text{a.e.}\  x\in \mathbb{R}^n.$$
Thus, $$|R_\alpha f(x)|=\lim_{\ell\to \infty}\left|\sum_{j=1}^{\psi(\ell)}\lambda_j R_\alpha a_j(x)\right|\le \sum_{j=1}^\infty |\lambda_j||R_\alpha a_j (x)|.$$ 
From this $$\|R_\alpha f\|_{L^1(\gamma)}\le \sum_{j=1}^\infty |\lambda_j|\|R_\alpha a_j\|_{L^1(\gamma)}\le C \sum_{j=1}^\infty |\lambda_j|,$$ where $C$ is an absolute constant independent of the atoms on which any $f\in X^k(\gamma)$ is decomposed as we shall see it in what follows. Therefore we get that $\|R_\alpha f\|_{L^1(\gamma)}\le C \|f\|_{X^k(\gamma)}.$

Let us then prove that 
\begin{equation}\label{eq: uniforme acotacion Riesz atomos}
\|R_\alpha a\|_{L^1(\gamma)}\leq C
\end{equation}
 for every $X^k$--atom $a$. For $k$ even, let us call $k=2j$. Then the ``old'' Gaussian Riesz transforms of even order now become
$\nabla^{2j} \mathcal{L}^{-j}$ with $\nabla^{2j} =\sum_{|\alpha|=2j}D^\alpha.$ 

Let $a$ be an $X^k$--atom. Then, by Remark~\ref{obs: contencion X^k atomos}, $a$ is also an $X^j$--atom. Since $a$ is supported in a critical ball $B$, by Proposition~\ref{propo: soporte de L-k y su norma 2} we have that $\mathcal{L}^{-j}a$ is supported in $\overline{B}$. By applying H\"older inequality and the boundedness of $\nabla^{2j} \mathcal{L}^{-j}$ on $L^2(\gamma)$ we get that
\[\|\nabla^{2j} \mathcal{L}^{-j} a\|_{L^1(\gamma)}\le \|\nabla^{2j} \mathcal{L}^{-j} a\|_{L^2(\gamma)}\gamma(B)^{1/2}\lesssim  \|a\|_{L^2(\gamma)}\gamma(B)^{1/2}\lesssim 1,\]
where we have also used Remark~\ref{obs: acotacion norma dos de atomo mas fuerte}.
This takes care of the boundedness from $X^k(\gamma)$ to $L^1(\gamma)$ of the ``old'' $k$-th order Gaussian Riesz transforms for $k$ even.

We now turn to the proof of the boundedness of this operator for $k$ odd. Given an $X^{k}$--atom $a$, notice that
\[\|R_\alpha a\|_{L^1(\gamma)}\leq \|R_\alpha a\|_{L^1(4B,\gamma)}+\|R_\alpha a\|_{L^1((4B)^c,\gamma)}.\]
For the first term on the right-hand side we apply again Hölder inequality, together with the $L^2(\gamma)$ boundedness of $R_\alpha$. This yields 
\[\|R_\alpha a\|_{L^1(4B,\gamma)}\leq \|R_\alpha a\|_{L^2(4B,\gamma)}\gamma(4B)^{1/2}\lesssim \| a\|_{L^2(\gamma)}\gamma(4B)^{1/2}\lesssim 1,\]
where in the last inequality we have used that $a$ is an $X^k$--atom and $\gamma$ is doubling over admissible balls. 

To take care of the term $\|R_\alpha a\|_{L^1((4B)^c,\gamma)}$, we write $D^\alpha \mathcal{L}^{-k/2}a=D^\alpha \mathcal{L}^{k/2}(\mathcal{L}^{-k}a).$
Since $a$ is an $X^k$--atom, we get $\supp \mathcal{L}^{-k}a\subseteq \overline{B}$  and 
\[\gamma(B)^{1/2}\|\mathcal{L}^{-k}a\|_{L^2(\gamma)}\lesssim r_B^{2k},\]
by Proposition~\ref{propo: soporte de L-k y su norma 2}. 
By applying Lemma~\ref{lem: lema para teo principal - 2} and the inequality above with ${f=\mathcal{L}^{-k}a}$, we get 
\begin{align*}
  \|D^\alpha \mathcal{L}^{-k/2}a\|_{L^1((4B)^c,\gamma)}&=\left\|D^\alpha \mathcal{L}^{k/2}f\right\|_{L^1((4B)^c,\gamma)}\\
  &\lesssim r_B^{-2k}\|f\|_{L^1(B,\gamma)}\\
  &\lesssim r_B^{-2k}\gamma(B)^{1/2}\|\mathcal L^{-k} a\|_{L^2(\gamma)}\\
  &\lesssim 1.
\end{align*}
This proves \eqref{eq: uniforme acotacion Riesz atomos} for every $k\in\mathbb N$.
\end{proof}

\medskip

\begin{proof}[Proof of Theorem~\ref{thm: teo principal 2}]
Let us point out here that in order to prove the boundedness of the new Gaussian Riesz transforms from $\mathcal{H}^1(\gamma)$ into $L^1(\gamma),$ 
 we will not be able to apply \cite[Theorem~6.1(ii) and Remark~6.2]{MM07} for if we consider $m(x,y)$ being the kernel associated to the old Gaussian Riesz transforms, $m^*(x,y)$ does not represent the kernel associated to the new higher order ones. So, given an $\mathcal{H}^1$--atom $a$, in order to prove this result we will proceed as in the proof of the boundedness of the old higher order Gaussian Riesz transforms of odd order, by splitting the norm in the following way
\begin{equation*}
\|R^*_\alpha a\|_{L^1(\gamma)}\le \|R^*_\alpha a\|_{L^1(2B,\gamma)}+\|R^*_\alpha a\|_{L^1((2B)^c, \gamma)}.
\end{equation*}
Then, the first term of the sum is bounded by the $L^2(\gamma)$-norm of $R^*_\alpha a$ times $\gamma^{1/2}(2B)$ and we use the continuity of $R_\alpha^*$ on $L^2(\gamma)$ and the fact that $\gamma$ is a doubling measure over $\mathscr{B}_1.$ 

Now, for taking care of the second term of the above sum we use the fact that the atom $a$ has its support contained in $B$ and it has average zero over that ball. This yields
\begin{align*}
\|R^*_\alpha a\|_{L^1((2B)^c,\gamma)}&\le C \int_B |a(y)| \int_{(2B)^c}\left|\int_0^1
\frac{\lambda_\alpha (r)}{(1-r^2)^{n/2+1}}(F_\alpha(x,y,r)-F_\alpha(x,c_B,r))dr\right| dxd\gamma(y)
\end{align*}
being $\lambda_\alpha(r)=\left(\frac{-\log r}{1-r^2}\right)^{\frac{|\alpha|}{2}-1}$ and 
$F_\alpha(x,y,r)=H_\alpha \left(\frac{x-ry}{\sqrt{1-r^2}}\right)e^{-\frac{|x-ry|^2}{1-r^2}}.$

By applying the mean value theorem to the function $F_\alpha$ in the variable $y$, we get that
$$\|R_\alpha^* a\|_{L^1((2B)^c,\gamma)}\lesssim \nu_s \|a\|_{L^1(\gamma)}\le \nu_s\|a\|_{L^2(\gamma)}\gamma^{1/2}(B)\lesssim \nu_s$$
being $$\nu_s=\sup_{B\in \mathscr{B}_1}\sup_{y\in B}r_B \int_{(2B)^c} s(x,y) dx$$ and \begin{align*}
   s(x,y)&
=c_n \int_0^1 \frac{\lambda_\alpha(r)}{(1-r^2)^{n/2+1}}|\nabla_y F_\alpha(x,y,r)| dr.
\end{align*}

Now we prove the finiteness of $\nu_s.$
%
For every $i=1,\dots,n$ \[\partial_{y_i}  F_\alpha(x,y,r)=2r\left(\alpha_iH_{\alpha-e_i}\left(\frac{x-ry}{\sqrt{1-r^2}}\right)-\frac{x_i-ry_i}{\sqrt{1-r^2}}H_\alpha\left(\frac{x-ry}{\sqrt{1-r^2}}\right)\right)\frac{e^{-\frac{|x-ry|^2}{1-r^2}}}{\sqrt{1-r^2}},\]
which leads to 
\[|\nabla_y F_{\alpha}(x,y,r)|\lesssim\left|P\left(\frac{x-ry}{\sqrt{1-r^2}}\right)\right|\frac{e^{-\frac{|x-ry|^2}{1-r^2}}}{\sqrt{1-r^2}},\]
where $P$ is a polynomial of degree $|\alpha|$. Therefore, 
\[|\nabla_y F_{\alpha}(x,y,r)|\lesssim\frac{e^{-\frac{|x-ry|^2}{2(1-r^2)}}}{\sqrt{1-r^2}}.\]
Consequently,
\begin{align*}
\nu_s&\lesssim \sup_{B\in \mathscr{B}_1}\sup_{y\in B} r_B \int_{(2B)^c}\int_0^1 \frac{\lambda_\alpha(r)}{(1-r^2)^{3/2}} \frac{e^{-\frac{|x-ry|^2}{2(1-r^2)}}}{(1-r^2)^{n/2}}dr dx  \\ &=\sup_{B\in \mathscr{B}_1}\sup_{y\in B} r_B \int_0^1 \frac{\lambda_\alpha(r)}{(1-r^2)^{3/2}}\int_{(2B)^c}\frac{e^{-\frac{|x-ry|^2}{2(1-r^2)}}}{(1-r^2)^{n/2}}dx dr
\end{align*}
In order to see that $\nu_s<\infty$ it will be enough to show that for every ball $ B\in \mathscr{B}_1$ and every $y\in B$,
\begin{equation*}
I=\int_0^1 \frac{\lambda_\alpha(r)}{(1-r^2)^{3/2}}\int_{(2B)^c}\frac{e^{-\frac{|x-ry|^2}{2(1-r^2)}}}{(1-r^2)^{n/2}}\,dx\,dr\lesssim \frac{1}{r_B}.
\end{equation*}

Let $r_{B,y}$ be the number defined in Lemma~\ref{lem: lema para teo principal - 1} and assume that $r_{B,y}\geq 1$. By applying items~\ref{item: lem: lema para teo principal - 1 - item 1} and~\ref{item: lem: lema para teo principal - 1 - item 3} in Lemma~\ref{lem: lema para teo principal - 1}, and the fact that $\lambda_\alpha(r)\leq C_{\alpha,n}r^{-1/2}$ for any $r\in (0,1/2)$, we get
\begin{align}
 I&\lesssim \int_0^1 \frac{\lambda_\alpha(r)}{(1-r^2)^{3/2}}e^{-C_2\frac{r_B^2}{1-r^2}}\,dr\nonumber\\
  &\lesssim \int_0^{1/2}\lambda_\alpha(r)\,dr+\int_{1/2}^1 \frac{e^{-\frac{C_2r_B^2}{2(1-r^2)}}}{(1-r^2)^{3/2}}\,dr\label{eq: estimacion exponencial}\\
  &\leq \sqrt{2} C_{\alpha,n}+\frac{2}{r_B}\int_{0}^\infty e^{-C_2u^2}\,du\nonumber\\
  &\lesssim  1+\frac{1}{r_B}\nonumber\\
  &\lesssim \frac{1}{r_B}.\nonumber
\end{align}

We turn now our attention to the case $r_{B,y}<1$. We split the integral $I$ into three parts, $I_1, I_2$ and $I_3$, being
\begin{align*}
I_1&=\int_0^{1/2}\frac{\lambda_\alpha(r)}{(1-r^2)^{3/2}}\int_{(2B)^c}\frac{e^{-\frac{|x-ry|^2}{2(1-r^2)}}}{(1-r^2)^{n/2}}\,dx\,dr,\\
I_2&=\int_{1/2}^{1-r_{B,y}}\frac{\lambda_\alpha(r)}{(1-r^2)^{3/2}}\int_{(2B)^c}\frac{e^{-\frac{|x-ry|^2}{2(1-r^2)}}}{(1-r^2)^{n/2}}\,dx\,dr,\\
I_3&=\int_{1-r_{B,y}}^{1}\frac{\lambda_\alpha(r)}{(1-r^2)^{3/2}}\int_{(2B)^c}\frac{e^{-\frac{|x-ry|^2}{2(1-r^2)}}}{(1-r^2)^{n/2}}\,dx\,dr.
\end{align*}
We shall estimate every term above separately.  We apply item~\ref{item: lem: lema para teo principal - 1 - item 2}  from Lemma~\ref{lem: lema para teo principal - 1} and the change of variable $z=|x-c_B|/\sqrt{1-r^{2}}$ to have
\[I_1\lesssim \int_0^{1/2}\lambda_\alpha(r)\int_{\mathbb{R}^n}\frac{e^{-\frac{|x-c_B|^2}{2(1-r^2)}}}{(1-r^2)^{n/2}}\,dx\,dr=\left(\int_0^{1/2}\lambda_\alpha(r)\,dr\right)\left(\int_{\mathbb{R}^n}e^{-\frac{|z|^2}{2}}\,dz\right)\lesssim 1\leq \frac{1}{r_B},\]
where we have again used the integrability of $\lambda_\alpha$ as in \eqref{eq: estimacion exponencial}.

On the other hand, since $\lambda_\alpha(r)\leq C_\alpha$ for $r\in [1/2,1]$, and $1-r^2\approx 1-r$ on $[1/2,1-r_{B,y}]$, by repeating the argument for $I_1$ on the inner integral, we get
\begin{align*}
I_2\lesssim\int_{1/2}^{1-r_{B,y}}\frac{1}{(1-r)^{3/2}}\int_{\mathbb{R}^n}\frac{e^{-\frac{|x-ry|^2}{2(1-r^2)}}}{(1-r^2)^{n/2}}\,dx\,dr&\lesssim \left(\int_{r_{B,y}}^{1/2} u^{-3/2}\, du\right)
\left(\int_{\mathbb{R}^n}e^{-\frac{|z|^2}{2}}\,dz\right)\\
&\lesssim \frac{1}{r_{B,y}^{1/2}}=\left(\frac{2|y|}{r_B}\right)^{1/2}\\
&\lesssim 1+\frac{1}{r_B}\lesssim \frac{1}{r_B},
\end{align*}
since $|y|^{1/2}\leq r_B^{1/2}+r_B^{-1/2}$.

Finally, the bound for $I_3$ follows as in the case $r_{B,y}\geq 1$ by using  item~\ref{item: lem: lema para teo principal - 1 - item 2} in Lemma~\ref{lem: lema para teo principal - 1} and proceeding similarly as in the estimate 
 for the second term of \eqref{eq: estimacion exponencial}.

Taking into account that these new higher order Gaussian Riesz transforms are also bounded from $L^1(\gamma)$ into $L^{1,\infty}(\gamma)$ (see \cite{AFS}) and proceeding as before when we have other types of continuity, these operators extend boundedly to the whole atomic space $\mathcal{H}^1(\gamma)$.
\end{proof}

%

\end{document}